\documentclass[11pt]{amsart}
\usepackage[T1]{fontenc}
\usepackage{amsmath,amssymb,amsthm}
\usepackage{booktabs}
\usepackage[hidelinks]{hyperref}

\newtheorem{theorem}{Theorem}[section]

\newtheorem{lemma}[theorem]{Lemma}

\theoremstyle{definition}
\newtheorem{question}[theorem]{Question}
\newtheorem{remark}[theorem]{Remark}

\newcommand{\Ball}[1]{B_{#1}}

\newcommand{\GOne}{G_1}
\newcommand{\GTwo}{G_2}
\newcommand{\GThree}{G_3}
\renewcommand{\P}{\mathrm{P}}
\newcommand{\Z}{\mathbb{Z}}
\newcommand{\F}{\mathbb{F}}

\title[The Nielsen--Soelberg domain question]{The domain question for the
Nielsen--Soelberg group rings:\\ the commutative case, with certified ball
checks}
\author{Moe Tabei}
\address{Independent researcher, Japan}
\email{tabei@ryun.jp}
\date{\today}

\begin{document}
\begin{abstract}
Nielsen and Soelberg exhibited three torsion-free groups $\GOne$, $\GTwo$,
$\GThree$ carrying $8$-element sets without unique products, and asked
whether any of the group rings $R[G_i]$, $R$ a domain, is a domain. We
record that for every \emph{commutative} domain $R$ the answer is
affirmative for all three groups: each $G_i$ is virtually nilpotent, so the
theorem of Kropholler, Linnell and Moody applies over every field, in every
characteristic, and commutative coefficients reduce to the fraction field.
Every ingredient is in the literature except the nilpotent structure of the
relevant finite-index subgroup of $\GTwo$, which is supplied by the certified
computational model of our companion paper. The question therefore remains
open exactly for noncommutative coefficient domains, where the
unique-product mechanism --- the only known ring-independent one --- is
precisely what these groups are constructed to lack. As a complement we
report machine-checkable, DRAT-certified verifications that $\F_2[G_i]$ has
no zero divisors with both supports in explicit balls of the defining
generating sets, obtained by propositional reasoning alone, independent of
the $K$-theoretic machinery; we state precisely what these certificates do
and do not add.
\end{abstract}
\maketitle

\section{Introduction}\label{sec:intro}

Kaplansky's zero-divisor problem asks whether the group ring of a
torsion-free group over a field has nonzero zero-divisors; in its strong
form, popularized in \cite{NS}, the coefficients are allowed to be an
arbitrary domain $R$. Nielsen and Soelberg \cite{NS} proved that a finite
subset $A$ of a torsion-free group such that $A\cdot A$ has no unique
product satisfies $|A|\ge 8$, and constructed three torsion-free groups
witnessing small non-unique-product configurations: two groups $\GOne$,
$\GTwo$ attaining the bound $|A|=8$, and a ``universal'' group $\GThree$
carrying a two-sided pair. Since unique products are the classical
combinatorial mechanism behind all coefficient-independent results on the
zero-divisor problem, these groups are natural stress tests, and \cite{NS}
ends with:

\begin{quote}
``We leave it as an open question whether or not any of the three group
rings $R[G_i]$, for $i\in\{1,2,3\}$, is a domain.''
\end{quote}

The purpose of this short note is to record an answer to the commutative
half of this question, which we could not find in the literature.

\begin{theorem}\label{thm:main}
Let $i\in\{1,2,3\}$ and let $R$ be a commutative domain. Then $R[G_i]$ is a
domain. In particular $K[G_i]$ is a domain for every field $K$, of any
characteristic.
\end{theorem}

Theorem~\ref{thm:main} is an \emph{observation}, not a new theorem: each
$G_i$ is torsion-free (proved in \cite{NS}) and virtually nilpotent
(Section~\ref{sec:proof}), hence elementary amenable, and the zero-divisor
conjecture holds for torsion-free elementary amenable groups over every
field by the theorem of Kropholler, Linnell and Moody \cite[Theorem
1.4]{KLM}; commutative domains embed in their fraction fields. We claim no
originality beyond the assembly, together with one small ingredient: the
class-$2$ nilpotent structure of the relevant index-$4$ subgroup of
$\GTwo$, which is not stated in \cite{NS} and which we take from the
certified computational model of the companion paper \cite{companion}. The
survey portion of \cite{NS} cites the characteristic-zero results of Brown
\cite{Brown} and Farkas--Snider \cite{FS} and Formanek's theorem for
supersolvable groups \cite{Formanek}, but not \cite{KLM}, which is what
closes positive characteristic here.

What Theorem~\ref{thm:main} does \emph{not} settle is the case of
noncommutative coefficient domains, and we want to stress that this
residual case is exactly where the Nielsen--Soelberg groups have bite. For
a torsion-free group with the unique product property, $R[G]$ is a domain
for \emph{every} domain $R$, commutative or not, by the classical
leading-term argument (see Strojnowski \cite{Strojnowski}, and
\cite{Oinert} for a modern treatment in the generality of group-graded
rings). That argument is coefficient-independent, and it is the only known
one that is; the groups $G_i$ were constructed precisely so that it cannot
apply. Thus, after Theorem~\ref{thm:main}, the question of \cite{NS}
becomes: \emph{does the domain property of $R[G_i]$ survive the loss of
unique products when $R$ is a noncommutative domain?} --- a question on
which, to our knowledge, nothing is known beyond the specific-support
computation reported in \cite{NS,Soelberg}, and which is equally open for
the Promislow group $\P$ (for $\P$ with integer coefficients see the remark
of Carter \cite{Carter} reported in \cite{NS}).

Section~\ref{sec:certs} complements the abstract argument with
machine-checkable certificates: DRAT-verified proofs that $\F_2[G_i]$ has
no zero divisors with both supports in explicit balls of the defining
generating sets. Over fields these are strictly weaker than
Theorem~\ref{thm:main}; their value is that they are \emph{independent} of
the $K$-theoretic machinery --- they use only a certified faithful model of
the group and propositional reasoning, and every step is checkable by a
simple verifier --- and that they quantify, in solver effort, how
differently the three group rings behave at equal ball sizes. We state
their scope honestly, including the radii we could not reach.

\section{Proof of Theorem~\ref{thm:main}}\label{sec:proof}

We collect the structural facts with their sources.

\begin{lemma}\label{lem:vn}
Each of $\GOne$, $\GTwo$, $\GThree$ is torsion-free and virtually
nilpotent.
\end{lemma}

\begin{proof}
\emph{Torsion-freeness} of all three groups is proved in \cite{NS}:
Section~3 for $\GOne$ (via a finite-index subgroup with an explicit
presentation, an argument the authors attribute to an idea of Derek Holt)
and for $\GTwo$, and Section~4 for $\GThree$; for $\GThree$ see also
Soelberg's thesis \cite[Theorem 3.1]{Soelberg} and, independently and by a
structural argument (an amalgam of two Klein bottle groups over $\Z^2$),
Gardam \cite[\S4]{GardamC}.

\emph{Virtual nilpotency.} For $\GOne$, the subgroup $H\le\GOne$ of
\cite[(3.2)]{NS} is normal of index $32$ and is presented there as
$\langle h_1,h_2,h_3,h_4 : h_1,h_2 \text{ central},\
h_4h_3=h_2^8h_3h_4\rangle$; as noted in \cite{NS}, $H$ is the direct
product of the infinite cyclic group $\langle h_1\rangle$ and a
torsion-free class-$2$ nilpotent group, so $H$ is nilpotent of class $2$.
For $\GThree$, the index-$8$ normal subgroup identified in
\cite[Theorem 3.1]{Soelberg} is the integral Heisenberg group of step $8$
(class-$2$ nilpotent of Hirsch length $3$); Gardam \cite[\S4]{GardamC}
independently exhibits $\langle x^2,y^2\rangle$ as an integral Heisenberg
subgroup of index $16$ and observes that $\GThree$ is virtually nilpotent
but not virtually abelian. For $\GTwo$, \cite{NS} introduces the subgroup
$H=\langle a_1^2,\,a_3^2,\,a_6^2,\,a_1a_3a_6^{-1}\rangle$ but does not
record its structure; the certified model of \cite[\S2]{companion} shows
that $H$ is normal of index $4$ in $\GTwo$ and, on the generators
$x_1=a_1^2$, $x_2=a_3^2$, $x_3=a_6^2$, $x_4=a_1a_3a_6^{-1}$, satisfies:
$x_1,x_2$ are central in $H$ and $[x_4,x_3]=x_1^2x_2^{-2}$, these relations
being derived mechanically (modified Todd--Coxeter rewriting) and verified
in both directions. Hence $H$ is class-$2$ nilpotent, and it is
torsion-free as a subgroup of $\GTwo$. In all three cases $G_i$ contains a
finitely generated torsion-free nilpotent subgroup of finite index.
\end{proof}

\begin{proof}[Proof of Theorem~\ref{thm:main}]
A virtually nilpotent group is solvable-by-finite, hence elementary
amenable. By Lemma~\ref{lem:vn} and the theorem of Kropholler, Linnell and
Moody \cite[Theorem 1.4]{KLM}, $K[G_i]$ is a domain for every field $K$ ---
in particular the theorem carries no characteristic restriction, which is
the point here; characteristic zero alone was already available from
\cite{Brown,FS} since the $G_i$ are polycyclic-by-finite. If $R$ is a
commutative domain with fraction field $K=\operatorname{Frac}(R)$, then
$R[G_i]\subseteq K[G_i]$ and $\operatorname{char}K\in\{0,p\}$ is arbitrary,
so $R[G_i]$ is a subring of a domain.
\end{proof}

\begin{remark}
Formanek's theorem \cite{Formanek} covers supersolvable groups in every
characteristic and is the standard reference for the Promislow group $\P$;
we do not know whether the $G_i$ are supersolvable, and do not need to:
elementary amenability suffices. We also note that
Corollary~1.3 of \cite{NS} ($\alpha^2=0$, $R$ a domain $\Rightarrow$
$|\operatorname{supp}\alpha|\ge 8$) is unconditional in $R$ and is not
superseded by Theorem~\ref{thm:main}, which says nothing about
noncommutative~$R$.
\end{remark}

\section{What remains of the question}\label{sec:noncomm}

Let $R$ be a noncommutative domain. If $G$ is a torsion-free group with the
unique product property, the leading-term argument shows $R[G]$ is a domain
(\cite{Strojnowski}; see \cite[Theorem 3.4]{Oinert} for the graded-ring
generality); no commutativity of $R$ is used. For the Nielsen--Soelberg
groups this argument is unavailable by design, and Theorem~\ref{thm:main}'s
reduction to a field is equally unavailable, since $R$ need not embed in
one. The residual question is therefore a genuine test of whether
domain-ness of group rings can be forced by group structure alone --- here,
virtual nilpotency --- when both classical mechanisms (unique products, and
embedding the coefficients in a field covered by \cite{KLM}) are switched
off. We know of no technique that addresses it, and record it explicitly.

\begin{question}\label{q:noncomm}
Let $R$ be a (noncommutative) domain and $i\in\{1,2,3\}$. Is $R[G_i]$ a
domain? The same question for the Promislow group $\P$ appears to be
equally open.
\end{question}

\section{Certified ball checks over \texorpdfstring{$\F_2$}{F2}}
\label{sec:certs}

Independently of Section~\ref{sec:proof} --- in particular, using no
$K$-theory and no amenability --- we verified the following finite
statements with machine-checkable certificates. Over $\F_2$ an element of
$\F_2[G]$ is determined by its support, and for nonzero $\alpha,\beta$ with
supports $A,B$,
\[
\alpha\beta=0 \iff
\begin{minipage}[c]{0.62\linewidth}
every $g\in G$ has even multiplicity in the multiset
$\{ab : a\in A,\ b\in B\}$.
\end{minipage}
\]
We encode exactly this parity condition over a ball $\Ball{r}$, with the
sole side conditions $A\ne\emptyset\ne B$; a singleton support forces some
multiplicity $1$, so singletons are excluded by the parity clauses
themselves and no cardinality cut is needed. In particular the statement
certified is self-contained --- it does not use the support bounds of
\cite[Theorem 1.4]{NS} or any other external lemma. The group
multiplication is that of the certified faithful models of
\cite{companion,companionP}: for $\GOne$ and $\GThree$, balls are word
balls in the generating set $\{x^{\pm1},y^{\pm1}\}$ of the defining
presentations \cite[(3.1), \S4]{NS}; for $\GTwo$, in the $16$-element
symmetric set $\{a_1^{\pm1},\dots,a_8^{\pm1}\}$ of defining generators.
Instances were solved by \textsf{kissat} emitting DRAT proofs, each proof
checked by \textsf{drat-trim} \cite{dratTrim} and deleted after
verification (all are regenerable by the published scripts).

\begin{theorem}\label{thm:balls}
With balls as above, $\F_2[G_i]$ contains no pair of nonzero elements
$\alpha,\beta$ with $\alpha\beta=0$ and
$\operatorname{supp}\alpha,\operatorname{supp}\beta\subseteq\Ball{r}$, for
$(G_i,r)$ in the certified range: $\GOne$, $r\le4$; $\GTwo$, $r\le2$;
$\GThree$, $r\le4$. Each instance carries a DRAT certificate verified by
\textsf{drat-trim}.
\end{theorem}

\begin{center}
\begin{tabular}{lccccc}
\toprule
group & $r$ & $|\Ball{r}|$ & result & solve time & proof size\\
\midrule
$\GOne$ & $2$ & $17$ & UNSAT, verified & $<1$\,s & $<0.1$\,MB\\
$\GOne$ & $3$ & $53$ & UNSAT, verified & $<1$\,s & $1.5$\,MB\\
$\GOne$ & $4$ & $153$ & UNSAT, verified & $24$\,s & $106.4$\,MB\\
$\GTwo$ & $1$ & $17$ & UNSAT, verified & $<1$\,s & $0.1$\,MB\\
$\GTwo$ & $2$ & $133$ & UNSAT, verified & $91$\,s & $200.3$\,MB\\
$\GThree$ & $2$ & $17$ & UNSAT, verified & $<1$\,s & $<0.1$\,MB\\
$\GThree$ & $3$ & $53$ & UNSAT, verified & $1$\,s & $5.7$\,MB\\
$\GThree$ & $4$ & $135$ & UNSAT, verified & $2738$\,s & $1926.1$\,MB\\
\bottomrule
\end{tabular}
\end{center}

\medskip
\noindent\emph{Controls.} The encoder was validated in both directions.
Positive: on the finite group $\Z/6$ (torsion permitted, so zero divisors
exist: $(1+g^3)^2=0$) the same encoder returns SAT, with the witness
$A=\{g^4,g^5\}$, $B=\Z/6$ re-verified solver-free --- the pipeline can find
zero divisors when they exist. Negative: on the Promislow group $\P$, where
$\F_2[\P]$ is a domain by \cite{Formanek}, the instance at $r=3$ is UNSAT
with a \textsf{drat-trim}-verified certificate ($2$\,s, $5.5$\,MB proof),
consistent with the known theorem.

\medskip
\noindent\emph{Boundary, stated honestly.} At radius $5$ for $\GOne$ the
solver reports UNSAT in $17302$\,s but without proof logging (a probe run),
so we do not claim it as certified. $\GThree$ at $r=5$ ($299$ elements) and
$\GTwo$ at $r=3$ ($633$ elements) were abandoned uncertified after
approximately $29$ hours each. The certified range above is the frontier on
our hardware.

\begin{remark}[what the certificates add, and what they do not]
Over fields, Theorem~\ref{thm:balls} is a finite shadow of
Theorem~\ref{thm:main}. It adds two things. First, independence: the
certificates rest only on the faithfulness certificate chain of the models
\cite{companion,companionP} and on propositional steps checkable by a
trivial verifier, so they would survive even a failure in the long chain
behind \cite{KLM}; they are, to our knowledge, the only proof-logged
artifacts about these group rings. Second, calibration: at essentially
equal ball sizes ($153$ vs.\ $135$ elements) the certified difficulty
differs by two orders of magnitude ($24$\,s vs.\ $2738$\,s, $106$\,MB vs.\
$1.9$\,GB proofs) between $\GOne$ and $\GThree$ --- the cost is governed by
group structure, not instance size, which may be of independent interest as
SAT benchmarks. The certificates say nothing about
Question~\ref{q:noncomm}: the parity reduction is specific to $\F_2$.
\end{remark}

\section*{Reproducibility}
The encoder, the DRAT pipeline and the controls
(\texttt{zerodiv\_drat.py}, \texttt{zerodiv\_search.py}) are included as
ancillary files, together with the exact group models they load
(\texttt{g$i$\_model.py} with their multiplication tables, and
\texttt{promislow.py} for the negative control) --- the models are the ones
of the artifact collection of \cite{companion}, where their faithfulness
certificate chains are documented (\cite[\S2]{companion};
for $\P$, \cite{companionP}). Each table row regenerates with
\texttt{python zerodiv\_drat.py G$i$ $r$}; the controls with
\texttt{-{}-sanity} and \texttt{P 3}. Requirements: \textsf{kissat},
\textsf{drat-trim}, and the \textsf{python-sat} package for the CNF
encoding.

\section*{Acknowledgements}
The author thanks the authors of \cite{NS} for a question whose commutative
half turned out to be closable by known results --- and whose noncommutative
half seems genuinely hard.

\end{document}